\documentclass[a4paper,11pt]{amsart}
\usepackage{amssymb}
\usepackage{amsmath}
\usepackage{mathrsfs}
\usepackage{amsmath,amssymb,amsthm,latexsym,amscd,mathrsfs}
\usepackage{indentfirst}
\usepackage{graphicx}
\usepackage{booktabs}
\usepackage{array}
\usepackage{chemarrow}
\DeclareMathOperator{\RE}{Re}

\newcommand{\Ca}{\mathbb{O}}  
\newcommand{\C}{\mathbb{C}}  
\newcommand{\kg}{\hspace{10pt}} 
\newcommand{\EJ}{\mathcal{H}_3}  
\newcommand{\AH}{\mathcal{SH}_3} 
\newcommand{\Tr}[1]{\mathrm{Tr}\hspace{1pt}(#1)} 

\newcommand{\R}{\mathbb{R}}
\newcommand{\ROM}[1]{\mathrm{\uppercase\expandafter{\romannumeral#1}}}
\theoremstyle{definition}
\newtheorem{exa}{Example}[section]
\newtheorem{thm}{Theorem}[section]
\newtheorem{lem}{Lemma}[section]

\newtheorem{rem}{Remark}[section]
\newtheorem{prop}{Proposition}[section]

\newtheorem{ack}{Acknowledgements}   

\title[Schoen-Yau-Gromov-Lawson theory and isoparametric foliations]{\textbf{Schoen-Yau-Gromov-Lawson theory and isoparametric foliations}}
\author[Z.Z.Tang]{Zizhou Tang}\address{School of Mathematical Sciences, Laboratory of Mathematics and Complex Systems, Beijing Normal
University, Beijing 100875, China}\email{zztang@bnu.edu.cn}
\thanks {The project is partially supported by the NSFC ( No.11071018 ) and the Program for Changjiang Scholars and Innovative
Research Team in University.}
\author[Y.Q.Xie]{Yuquan Xie}
\address{Department of Mathematics, Hangzhou Normal University, Hangzhou, Zhejiang 310036, China; School of Mathematical Sciences, Peking University, Beijing 100871, China} \email{yuqxie@pku.edu.cn}
\author[W. J. Yan]{Wenjiao Yan}
\address{School of Mathematical Sciences, Laboratory of Mathematics and Complex Systems, Beijing Normal
University, Beijing 100875, China} \email{wjyan@mail.bnu.edu.cn}

 \subjclass[2000]{ 53C12, 53C40, 57R20.}
\date{}
\keywords{Schoen-Yau-Gromov-Lawson theory, isoparametric hypersurface, positive scalar curvature, double of a manifold, cohomogeneity one action.}
\begin{document}

\maketitle
\begin{center}
Dedicated to Professor Yuanlong Xin on his 70th birthday.
\end{center}
\begin{abstract}
Motivated by the celebrated Schoen-Yau-Gromov-Lawson surgery theory
on metrics of positive scalar curvature, we construct a double
manifold associated with a minimal isoparametric hypersurface in the
unit sphere. The resulting double manifold carries a metric of
positive scalar curvature and an isoparametric foliation as well. To
investigate the topology of the double manifolds, we use K-theory
and the representation of the Clifford algebra for the FKM-type, and
determine completely the isotropy subgroups of singular orbits for
homogeneous case.
\end{abstract}

\section{Introduction}

One of the simplest invariants of a Riemannian manifold is its
scalar curvature function. Here, we say an $n$-dimensional manifold
$M$ carries a metric of positive scalar curvature $R_M$ if $R_M\geq
0$ and $R_M(p) > 0$ for some point $p\in M $. Then a natural
question to raise is ``Which manifolds admit Riemannian metrics of
positive scalar curvature?" In recent decades, this subject has been
the focus of lively research. The first important contribution to
this subject was made by A. Lichnerowicz (\cite{Lic}) in 1962, who
showed that a compact spin manifold with non-vanishing
$\widehat{A}$-genus cannot carry a Riemannian metric of positive
scalar curvature. N. Hitchin (\cite{Hit}) generalized this result.
More precisely, he used a ring homomorphism $\alpha$, constructed by
J. Milnor, from $\Omega_*^{\text{spin}}$, the spin cobordism ring,
to $KO^{-*}(\text{pt})$, and proved that $\alpha(M)$ vanishes if $M$
carries a metric of positive scalar curvature. When $\dim M \equiv
0~ (\mod 4)$, $\alpha(M)$ can be identified with $\widehat A(M)$ (up
to a factor), so this recovers the result of Lichnerowicz.
One surprising and
beautiful result of this study was that half of the exotic spheres
in dimensions $8k+1$ and $8k+2$ cannot carry metrics of positive
scalar curvature. Another remarkable step toward answering the
question above was made when Schoen and Yau (\cite{SY}), and
independently, Gromov and Lawson (\cite{GL1}) established the
following ``surgery theorem" on metrics of positive scalar
curvature:

\noindent
\textbf{Theorem}\,\,
{\itshape
Let $X$ be a compact manifold which carries a Riemannian metric of positive scalar curvature. Then any manifold which can
be obtained from $X$ by performing surgeries in codimension $\geq3$ also carries a metric of positive scalar curvature.
}

Inspired by Schoen-Yau-Gromov-Lawson's surgery theory, we will construct a new manifold with rich geometrical properties from
a Riemannian manifold with an embedding hypersurface. In particular, we implement this construction on a unit sphere with
a minimal isoparametric hypersurface, finding that the new manifold admits not only complicated topology, but also a metric of
positive scalar curvature. Moreover the isoparametric foliation is kept. Details of the construction are given in the following.

Given a compact, connected manifold $X^n$ ($n \geq 3$) without boundary. Let $Y^{n-1}\hookrightarrow X^n$ be a
connected embedding hypersurface with a trivial normal bundle, and $\pi_0(X- Y) \neq 0$, \emph{i.e.},
the complement of $Y$ in $X$ is not connected. Then $Y^{n-1}$ separates $X^n$ into two components,
say $X^{n}_{+}$ and $X^n_{-}$, with the same boundary $Y^{n-1}$. (The assumption $\pi_0(X- Y) \neq 0$ is necessary.
For example, $T^2-S^1$, removing a latitude circle from the torus, is connected.) Since $Y$ has a trivial normal bundle in $X$, we can choose a unit normal vector field $\xi$
on $Y$, which is an interior normal direction with respect to $X_+$.

Define a continuous function $r: X^n \longrightarrow \mathbb{R}$
\begin{displaymath}
x \mapsto \left\{ \begin{array}{ll}
~~\mathrm{dist}(x, Y), & \textrm{if $x \in X_+$}\\
-\mathrm{dist}(x, Y), & \textrm{if $x \in X_-$}
\end{array} \right.
\end{displaymath}
where $\mathrm{dist}(x, Y)$ means the distance from $x$ to $Y$.
Clearly, $X_+$ ($X_-$) is just the subset that $r\geq 0$ (resp.
$r\leq 0$). Let $Y_r:=\{x \in X|~ r(x)=r\}$ for $|r|$ so small that
$Y_r$ is still an embedding hypersurface. We extend $\xi$ to a unit
vector field in a neighborhood of $Y$ such that $\xi$ is normal to
$Y_r$.

From now on, without loss of generality, we only deal with $X^n_+$. Concerning with the Riemannian product space $X^n_+\times \mathbb{R}$ with coordinates $(x, t)$,
for a small number $\bar{r}>0$, we can define a hypersurface $M^n$ in $X^n_+\times \mathbb{R}$ as \cite{GL1}
$$M^n := \{(x,t)\in X^n_+\times \mathbb{R}~:~(r(x), t)\in \gamma, ~r(x)\leq \bar{r}\}$$
where $\gamma$ is a curve in the $(r,t)$-plane as pictured below:
\begin{figure}[h]
\label{N}
\begin{center}
\includegraphics[height=46mm]{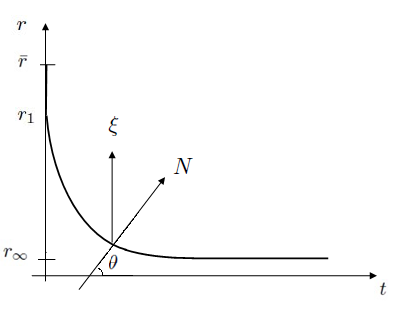}
\end{center}
\end{figure}

$$N:~unit~``exterior"~normal~vector~field~on~M,~~\sin\theta:=\langle N, \xi\rangle$$

The curve $\gamma$ begins at one end with a vertical line segment $t=0$, $r_1\leq r\leq \bar{r}$, and ends with a horizontal line segment
$r=r_\infty >0$, with $r_\infty$ small enough as we will require.

Now fix a point $q=(x,t)\in M$ corresponding to $(r, t)\in \gamma$.
Choose an orthonormal basis $e_1, e_2,...,e_{n-1}$ of $T_xY_r$ such
that the shape operator $A_{\xi}$ is expressed as
$A_{\xi}e_i=\mu_i(r)e_i$ for $i=1,...,n-1$. Then the associated
principal curvatures of $M$ at $q$ are of the form
$\lambda_i=\mu_i(r) \sin\theta$ for $i=1,...,n-1$. As observed by
Gromov-Lawson (\cite{GL1}), the tangent vector of $M\cap (l\times
\mathbb{R})$ is also a principal direction for the second
fundamental form of $M$ in $X^n_+\times \mathbb{R}$, where $l$ is a
geodesic ray in $X_+$ emanating from $Y$. We denote this tangent
vector by $e_n$. Thus the $n$-th principal curvature at $q$ is
$\lambda_n:= k$, the (nonnegative) curvature of $\gamma$ at $(r,
t)$.

Look at the Gauss equation:
$$K_{ij}^M=K_{ij}^{X\times\mathbb{R}} + \lambda_i\lambda_j,\qquad1\leq i,j\leq n,$$
where $K_{ij}^M$ is the sectional curvature of $M$ of the plane $e_i\wedge e_j$, and $K_{ij}^{X\times\mathbb{R}}$
is the corresponding sectional curvature of $X_+\times\mathbb{R}$. Since the metric of $X_+\times\mathbb{R}$ is the product metric,
we see:
\begin{eqnarray*}
&&K_{ij}^{X\times\mathbb{R}}=K_{ij}^X, \qquad1\leq i,j\leq n-1,\\
&&K_{n,j}^{X\times\mathbb{R}}=K_{\xi,j}^X\cos^2\theta,
\end{eqnarray*}
where $K^X$ is the sectional curvature of $X_+$.

It follows immediately that the scalar curvature of $M$ with the induced metric can be expressed as:
\begin{equation}\label{scalar}
R_M=\sum_{i\neq j}^nK_{ij}^M =R_X+2A\sin^2\theta+2kH(r)\sin\theta
\end{equation}
where $$A:=\sum_{i< j\leq n-1}\mu_i(r)\mu_j(r) -Ric^X(\xi,\xi);\quad H(r)=\sum_{i=1}^{n-1}\mu_i(r), \emph{ ~the ~mean~ curvature~
of}~ Y_r.$$

\begin{rem}\label{scalar1}
Since $\gamma$ ends with a horizontal line segment, $M$ has the standard product metric as $t$ goes to infinity. This guarantees
that we can glue $X_+$(\emph{resp}.$X_-$) smoothly in metric with a copy of itself along $Y$ to get a new manifold called the double
of $X_+$ (\emph{resp}.$X_-$), and denoted by $D(X_+)$(\emph{resp}.$D(X_-)$). The double of a manifold with boundary, as a topological
concept, appeared in 1930's. Gromov-Lawson (\cite{GL2}) studied the geometric property of the double of a manifold, and
showed an interesting theorem which states that if $X$ carries a metric of positive scalar curvature, and $Y$ is a minimal hypersurface,
then the double manifold $D(X_+)$(\emph{resp}.$D(X_-)$) also carries a metric of positive scalar curvature. But in their
construction of the double manifold, they ``bent" too much near the boundary of $X_+$, inducing some singularities or creases.
However in our method, an explicit construction of $D(X_+)$(\emph{resp}.$D(X_-)$) with satisfactory properties is given.


\end{rem}

\begin{rem}\label{scalar2}
Formula (\ref{scalar}) is the expression of the scalar curvature of $M$ in $X_+\times\mathbb{R}$. It also holds in $X_-\times\mathbb{R}$
although $\xi$ is the exterior normal vector field of $X_-$.
Gromov and Lawson(\cite{GL1}) studied the scalar curvature of $M$. Their formula is expressed in form of the estimate of principal
curvatures, while ours is an explicit expression. Their main result on surgery is of course correct although in their formula $(1)$ (\cite{GL1}) they lost a factor $2$. In addition, in $(1')$ they missed one item related to the second fundamental form of the submanifold. But this mistake would result in the missing of the item $H(r)$ in our formula (\ref{scalar}), which is, however, essential for our research.
Rosenberg and Stolz(\cite{RS}) modified Gromov-Lawson's expression, but they also lost
the principal curvatures or the second fundamental form of the submanifold.

\end{rem}

From now on, we will be concerned with $X^n=S^n(1)$. Suppose that
$Y^{n-1}$ is a compact minimal isoparametric hypersurface,
\emph{i.e.}, a compact hypersurface with vanishing mean curvature
and constant principal curvatures $\mu_1, \mu_2,...,\mu_{n-1}$
(\emph{cf. }\cite{CR}). In fact, in every isoparametric family in
the unit sphere, there does exist one and only one minimal
hypersurface (\emph{cf.} \cite{GT1}). Recall an elegant result of
M\"{u}nzner (\cite{Mu80}) that the number $g$ of distinct principal
curvatures must be $1$, $2$, $3$, $4$ or $6$; and the multiplicities
$m_i$ $(i=1, 2,..., g)$ of distinct principal curvatures satisfy
$m_k=m_{k+2}$ (subscripts $\mod ~g$). We will denote them by $m_+$
and $m_-$, respectively. More precisely, as it is well known that
every isoparametric hypersurface in the unit sphere corresponds to
an isoparametric function $f$ with image $[-1, 1]$. Denote the focal
submanifolds by $M_+:=f^{-1}(1)$ and $M_-:=f^{-1}(-1)$ so that
$codim(M_+)=m_++1$, $codim(M_-)=m_-+1$.

One of the main results of the present paper is:
\begin{thm}\label{isop1}
Let $Y^{n-1}$ be a compact minimal isoparametric hypersurface in $S^n(1)$, $n\geq 3$, which separates $S^n$ into $S^n_+$ and $S^n_-$.
Then each of doubles $D(S^n_+)$ and $D(S^n_-)$ admits a metric of positive scalar curvature. Moreover, there is still an
isoparametric foliation in $D(S^n_+)$ (or $D(S^n_-)$).
\end{thm}

\begin{rem}
As a direct result, we get the $KO$-characteristic numbers
$$\alpha(D(S^n_+))=0,~~~~~~\alpha(D(S^n_-))=0.$$

\end{rem}

Furthermore, we have:

\begin{prop}\label{pi-manifold}
$D(S^n_+)$ (\emph{resp}.$D(S^n_-)$) is a $\pi$-manifold, \emph{i.e.}, a stably parallelizable manifold. In
particular, it is an orientable, spin manifold with vanishing Stiefel-Whitney
classes and Pontrjagin classes.
\end{prop}

It is worth pointing out that the condition ``$D(S^n_+)$ is stably
parallelizable" does not imply the conclusion $\alpha(D(S^n_+))=0$.
For instance, as Kervaire-Milnor ( Theorem 3.1 in \cite{KM}) proved,
every homotopy sphere is a $\pi$-manifold. But as we stated before,
there do exist some $8k+1$ and $8k+2$ dimensional exotic spheres
with non-vanishing $KO$-characteristic number $\alpha$. \vspace{4mm}

For isoparametric hypersurfaces in unit spheres, taking the different values of $g$ into account, we know:

When $g = 1$, an isoparametric hypersurface must be a great or small
sphere. Thus the double construction is trivial, namely, $D(S^n_+) \cong S^n$.

When $g = 2$, an isoparametric hypersurface must be a standard product of two spheres $S^k(r)\times S^{n-k-1}(s)$ with $r^2+s^2=1$.
Thus $D(S^n_+) \cong S^k \times S^{n-k}$ or $S^{k+1} \times S^{n-k-1}$.

When $g = 3$, E. Cartan has classified the isoparametric hypersurfaces. In fact, they are all homogeneous (see, for example, \cite{CR}).

When $g = 4$, except for the unknown case $(m_+, m_-)=(7, 8)$ (or $(8, 7)$), all the isoparametric hypersurfaces are either FKM-type or homogeneous
(\emph{cf}. \cite{CCJ}, \cite{Chi}).

When $g = 6$, all the isoparametric hypersurfaces must be homogeneous (see, for example, \cite{Miy09}).

\vspace{3mm}

Given all these classifications, in order to study the properties of
the double manifold $D(S^n_+)$, it suffices to consider the cases
that $Y$ is either homogeneous or of the FKM-type, except for the
case $(g, m_+, m_-)=(4, 7, 8)$. Therefore, we divide our research
into two parts, one is on the homogeneous case, and the other is on
the FKM-type. \vspace{3mm}

We begin by recalling a well known result that homogeneous hypersurfaces in $S^n$ are isoparametric since they have constant
principal curvatures. They have been characterized as principal orbits of the isotropy representation of rank two
symmetric spaces, and are classified completely by Hsiang and Lawson (\emph{cf.} \cite{HL}, \cite{TT}). From the corresponding
cohomogeneity one action on
$S^n$ with a certain slice representation of the normal disc, we derive a cohomogeneity one action on $D(S^n_+)$. In terms of
the isotropy subgroup $K_0$ of the principal orbit and $K_\pm$ of the singular orbits (focal submanifolds) $M_\pm$,
we classify $D(S^n_+)$ in Section $3$ with respect to the homogeneous hypersurface $Y$.

In particular, we investigate $D(S^4_+)$ in the case $(g,m_+,m_-)=(3,1,1)$, finding an interesting phenomenon that
$D(S^4_+)\cong S^2\times S^2/\sigma$,
where $\sigma$ is an involution different from that of the oriented Grassmannian
$G_2(\mathbb{R}^4)\cong S^2\times S^2/\sim $.
\vspace{3mm}

Next, we turn to the FKM-type. For every orthogonal representation
of the Clifford algebra $\mathcal{C}_{m-1}$ on $\mathbb{R}^l$,
Ferus, Karcher and M\"{u}nzner (\cite{FKM}) constructed a
homogeneous polynomial $F$ on $\mathbb{R}^{2l}$. The level
hypersurfaces of $F|_{S^{2l-1}}$ are isoparametric in $S^{2l-1}$
with $g=4$ and multiplicities of distinct principal curvatures
$(m_+,m_-,m_+,m_-)=(m,l-m-1,m,l-m-1)$. If $m\not \equiv 0 $
$(\mod~4)$, $F$ is determined by $m$ and $l$ up to a rigid motion of
$S^{2l-1}$; if, however $m\equiv 0$ $(\mod~4)$, there are
inequivalent representations of $\mathcal{C}_{m-1}$ on
$\mathbb{R}^l$ parameterized by an integer $q$, the index of the
representation (\emph{cf}. \cite{Wan}). In the second case, denote
by $M_+(m,l,q)$, $M_-(m,l,q)$ the corresponding focal submanifolds,
respectively. According to \cite{FKM}, $M_+$ has a trivial normal
bundle, while $M_-$ is diffeomorphic to an $S^{l-1}$ bundle over
$S^m$. Thus $D(S^{2l-1}_+)\cong M_+\times S^{m+1}$. As for
$D(S^{2l-1}_-)$, a delicate calculation of the topology on a sphere
bundle over $M_-$ leads to the following

\begin{thm}\label{FKM}
Given an odd prime $p$. If $q_1\not \equiv \pm q_2$ $(\mod~p)$, then
$D(S^n_-)(m,l,q_1)$ and $D(S^n_-)(m,l,q_2)$ have different homotopy
types.
\end{thm}

As we claimed in Proposition \ref{pi-manifold}, $D(S^n_+)$ is a
$\pi$-manifold with vanishing Stiefel-Whitney classes and Pontrjagin
classes. Without the aid of characteristic classes, it is usually
not easy to distinguish the homeomorphism classes. Our Theorem
\ref{FKM} is established by using $\mod~p$ cohomology operators.


\section{Geometry of the double manifold $D(S^n_+)$}

This section will be devoted to the proof of Theorem \ref{isop1}. We prefer to prove this result by making use of
fundamental properties of isoparametric hypersurfaces and some straightforward verifications.

Let $Y^{n-1}$ be a minimal isoparametric hypersurface in the unit sphere $S^n(1)$. It is well known that $Y$ is a
level hypersurface with vanishing mean curvature of an isoparametric function $f$ on $S^n(1)$. By an isoparametric function on $S^n(1)$,
we mean a function $f: S^n(1)\rightarrow \mathbb{R}$
satisfying:
\begin{equation}\label{ab}
\left\{ \begin{array}{ll}
|\nabla f|^2= b(f)\\
\quad\triangle f=a(f)
\end{array}\right.
\end{equation}
where $\nabla f$ is the gradient of $f$, $\triangle f$ is the
Laplacian of $f$, $b$ is a smooth function on $\mathbb{R}$, and $a$ is a continuous function on $\mathbb{R}$
(see \cite{Tho}, for an excellent survey).
We require that the isoparametric function is proper (\emph{cf}. \cite{GT2}) so that both focal submanifolds have
codimensions greater than $1$.

Recall that an isoparametric hypersurface $Y^{n-1}$ in $S^n(1)$ has constant principal curvatures, which
we denote by $\mu_1(0), \mu_2(0),...,\mu_{n-1}(0)$ as before corresponding to the unit normal vector field $\xi=\frac{\nabla f}{|\nabla f|}$.
A key reason for choosing $Y^{n-1}$ to be minimal isoparametric is that, as we will see, its induced metric
from $S^n(1)$ has positive scalar curvature.

By Gauss equation, for a closed minimal hypersurface $N$ in a unit sphere $S^n(1)$, $$S=(n-1)(n-2)-R_N$$ where $S$ is square of the length of
the second fundamental form. If, in addition, $N$ is a minimal isoparametric hypersurface on $S^n$, Peng and Terng (\cite{PT}) asserted that:
$$S=(g-1)(n-1),$$
which implies $R_N\geq 0$, and ``=" is achieved if and only if $(m_+, m_-)=(1,1)$ since $n-1=\frac{g}{2}(m_++m_-)$.

\vspace{5mm}

It follows immediately that the minimal isoparametric hypersurface $Y$ has $R_Y\geq 0$, $H(0)={\displaystyle\sum_{i=1}^{n-1}\mu_i(0)}=0$, and $S=\displaystyle\sum_{i=1}^{n-1}\mu_i^2(0)=(g-1)(n-1)$,
which imply that $$2\sum_{i< j}^{n-1}\mu_i\mu_j|_Y=H(0)^2-\sum_{i=1}^{n-1}\mu_i^2|_Y=-(g-1)(n-1).$$
By definition in formula $(\ref{scalar})$, we see $A=\displaystyle\sum_{i< j\leq n-1}\mu_i\mu_j -(n-1)$. In order to simplify the
calculation of $R_M$, we set
$$a(r):=2\sum_{i<j}^{n-1}\mu_i\mu_j|_{Y_r}-2\sum_{i<j}^{n-1}\mu_i\mu_j|_{Y}=2\sum_{i<j}^{n-1}\mu_i\mu_j|_{Y_r}+(g-1)(n-1).$$
Since
\begin{displaymath}
\lim_{r\rightarrow 0}2\sum_{i<j}^{n-1}\mu_i\mu_j|_{Y_r}=\lim_{r\rightarrow 0}\Big(H(r)^2-\sum_{i=1}^{n-1}\mu_i^2|_{Y_r}\Big)=-(g-1)(n-1),
\end{displaymath}
we have
\begin{equation}\label{a(r)}
\lim_{r \to 0}a(r)=0
\end{equation}

In fact, according to the Bochner-Weitzenb\"{o}ck formula:
\begin{equation*}
\frac{1}{2}\triangle|\nabla f|^2=|Hess f|^2+\langle\nabla f, \nabla(\triangle f)\rangle + Ric(\nabla f, \nabla f),
\end{equation*}
by virtue of the expression of Hessian of $f$ (\emph{cf}. \cite{GTY}):
\begin{equation*}
Hess f=diag\Big(-\sqrt{b(f)}\mu_1,\cdots,-\sqrt{b(f)}\mu_{n-1},
b'(f)/2\Big),
\end{equation*}
with $b(f)=g^2(1-f^2)$ in formula (\ref{ab}) by the famous Cartan-M\"{u}nzner equations, we obtain
\begin{equation}\label{sum of mu}
\sum_{i=1}^{n-1}\mu_i^2|_{Y_r}=\frac{(n-1)(g-1)-cf+(n-1)f^2}{1-f^2}\quad \emph{with}\quad c=\frac{g^2(m_--m_+)}{2}.
\end{equation}
Hence we can express $H(r)$ explicitly as:
\begin{equation}\label{Hr}
H(r)=(n-1)\frac{f|_{Y_r}}{\sqrt{1-f^2|_{Y_r}}}-\frac{c}{g\sqrt{1-f^2|_{Y_r}}}.
\end{equation}
It follows immediately that $H(0)=0$ and $H(r)>0$ for any $r>0$.

Consequently, from the definition of $a(r)$ and (\ref{sum of mu}), (\ref{Hr}),
it follows that
\begin{eqnarray}\label{exact a(r)}
&&a(r)\\
&=&~H^2(r)-\sum_{i=1}^{n-1}\mu_i^2|_{Y_r}+(g-1)(n-1)\nonumber\\
&=&~(g-1)(n-1)+\frac{1}{1-f^2}\{(n-1)^2f^2-\frac{2c}{g}(n-1)f+\frac{c^2}{g^2}-(n-1)f^2+cf\nonumber\\
& &\quad-(g-1)(n-1)\},\nonumber
\end{eqnarray}

Substituting all these equalities in $(\ref{scalar})$, we get immediately
\begin{equation}\label{scalarYr}
R_{M}|_{Y_r}=n(n-1)\cos^2\theta+(n-g-1)(n-1)\sin^2\theta+a(r)\sin^2\theta+2kH(r)\sin\theta
\end{equation}
with $H(r)$ and $a(r)$ in (\ref{Hr}), (\ref{exact a(r)}), respectively.
\vspace{4mm}

Since we have the dimension relation $n-1=\frac{g}{2}(m_++m_-)$, it suffices to analyze
the following two cases for our destination.
\vspace{4mm}

\noindent
\emph{$(A)$}: $(m_+, m_-)=(1, 1)$.

This is just the case that $n-g-1=0$. Since $a(r)$ is identically $0$ in this case, (\ref{scalarYr}) becomes
$$R_M=n(n-1)\cos^2\theta+2kH(r)\sin\theta.$$
By controlling the ``bending" angle of the curve $\gamma$, we can
assume $0\leq k\leq \frac{1}{2}$ so that
$R_M|_{Y_r}=n(n-1)\cos^2\theta+2kH(r)\sin\theta \geq 0,$ and ``=" is
achieved if and only if $r=0$. \vspace{6mm}

\noindent
\emph{$(B)$}: $Max\{m_+, m_-\}\geq 2$.

In this case, $n-g-1>0$. For $\theta \in [0,\frac{\pi}{2}]$, it is easily seen that
$$Min\{n(n-1)\cos^2\theta+(n-g-1)(n-1)\sin^2\theta\}=(n-g-1)(n-1),$$ thus by (\ref{scalarYr}),
$$R_M\geq (n-g-1)(n-1) + a(r)\sin^2\theta + 2kH(r)\sin\theta.$$
With the same assumption on $k$ as in case $(A)$, $R_M$ has a positive lower bound.

Up to now, we changed only the metric near the minimal isoparametric
hypersurface $Y$ along the curve $\gamma$ into a product metric
while preserving the positive scalar curvature, as desired. In this
way, gluing two copies of $S^n_+$, we get the double manifold of
positive scalar curvature. More importantly, there is still an
isoparametric foliation on $D(S^n_+)$, remaining the same with that
in $S^n(1)$ as $r\geq r_1$. In a neighborhood of $Y$ with diameter
$2r_1$, the principal curvatures turn out to be $\mu_1\cos\theta,
\mu_2\cos\theta,...,\mu_{n-1}\cos\theta$.

The proof is now complete.\hfill $\Box$


\section{Topology of the double manifold $D(S^n_+)$}

First of all, we compute the cohomology groups.
\begin{prop}\label{homology}
Let the ring of coefficient $R=\mathbb{Z}$ if $M_+$ and $M_-$ are both orientable and $R=\mathbb{Z}_2$ otherwise. Then
\begin{displaymath}
\left\{ \begin{array}{ll}
H^0(D(S^n_+))=R & \textrm{}\\
H^1(D(S^n_+))=H^1(M_+) & \textrm{}\\
H^q(D(S^n_+))=H^{q-1}(M_-) \oplus H^q(M_+) & \textrm{ for $2 \leq q \leq n-2$ }\\
H^{n-1}(D(S^n_+))=H^{n-2}(M_-)\\
H^n(D(S^n_+))=R
\end{array} \right.
\end{displaymath}
For $D(S^n_-)$, analogous identities hold.\hfill $\Box$
\end{prop}

\begin{rem}\label{ori}
By Morse theory, we see that if $m_+>1$ (\emph{resp.}~$m_->1$), then $M_-$ (\emph{resp.}~$M_+$) is orientable.
In fact, we define a spherical distance function on the focal submanifold $M_-$.
\begin{eqnarray*}
&&L_p:~M_-\longrightarrow \mathbb{R}\\
&&\qquad\quad x\mapsto \cos^{-1}\langle p, x\rangle
\end{eqnarray*}
where $p$ belongs to the complement of $M_{\pm}$ in $S^n$. The Morse index theorem states that the index of $L_p$ at
a non-degenerate critical point $x$ equals the number of focal points (counting multiplicities ) of $(M_-, x)$ on
the shortest geodesic segment from $p$ to $x$. Immediately, we obtain, for example when $g=4$, the index of non-degenerate
critical points are $0$, $m_+$, $m_++m_-$ and $2m_++m_-$, respectively. Consequently, we have the cell decomposition $M_-=S^{m_+}\bigcup e^{m_++m_-}\bigcup e^{2m_++m_-}$.
Thus if $m_+>1$, $M_-$ is simply connected. Similar results hold for other values of $g$.
\end{rem}

In order to prove Proposition \ref{homology}, we recall
a topological theorem of M\"{u}nzner (\emph{cf.} \cite{Mu80}) stated as

\noindent
\textbf{Theorem}\,\,
{\itshape
Let $N$ be a compact connected hypersurface in $S^{n}$ such that:

$(a)$ $S^{n}$ is divided into two manifolds $B_+$ and $B_-$ with the same boundary $N$.

$(b)$ $B_+$ $(resp.~ B_-)$ has the structure of a disc bundle over a compact manifold $M_+$ $(resp.~ M_-)$
of dimension $n-1-m_+$ $(resp.~ n-1-m_-)$.

Let the ring of coefficient $R=\mathbb{Z}$ if $M_+$ and $M_-$ are both orientable and $R=\mathbb{Z}_2$, otherwise. Let $\nu=m_+ + m_-$.
Then
\begin{displaymath}
H^q(M_\pm) = \left\{ \begin{array}{ll}
R, & \textrm{for $q\equiv 0$ $(\mod ~\nu)$, $0 \leq q < n-1$}\\
R, & \textrm{for $q\equiv m_\mp$ $(\mod ~\nu)$, $0 \leq q < n-1$}\\
0, & \textrm{otherwise }
\end{array} \right.
\end{displaymath}

Further,
\begin{displaymath}
H^q(N) = \left\{ \begin{array}{ll}
R, & \textrm{for $q=0, n-1$}\\
H^q(M_+) \oplus H^q(M_-), & \textrm{for $1 \leq q \leq n-2$}
\end{array} \right.
\end{displaymath}

}\hfill $\Box$

To complete the proof of Proposition \ref{homology}, we observe that a
minimal isoparametric hypersurface $Y$ in $S^n$ satisfies the hypotheses of the previous
theorem, getting the cohomology groups $H^q(M_\pm)$, equivalently, $H^q(S^n_\pm)$. Finally, by the Mayer-Vietoris sequence
of $(D(S^n_+), S^n_+, S^n_+)$, we arrive at the conclusion immediately.\hfill $\Box$
\vspace{3mm}

Next, we give a proof of

\noindent
\textbf{Proposition 1.1}\,\,
{
$D(S^n_+)$ (\emph{resp}.$D(S^n_-)$) is a $\pi$-manifold, \emph{i.e.}, a stably parallelizable manifold. In
particular, it is an orientable, spin manifold with vanishing Stiefel-Whitney
classes and Pontrjagin classes.
}

Suppose we are now given a (minimal) isoparametric hypersurface in $S^n$. As M\"{u}nzner asserted (\emph{cf}. \cite{CR} p.283),
$S^n_+$ has the structure of a differential disc bundle over
$M_+$. In fact, it is the normal disc bundle over $M_+$. More precisely, we have
\begin{eqnarray}\label{normal disc bundle}
&& B^{m_++1} \hookrightarrow S^n_+ =B(\nu_+)\\
&&\qquad\qquad\quad \downarrow \pi\nonumber\\
&&\qquad\quad\quad\quad M_+\nonumber
\end{eqnarray}
where $\nu_+$ is the normal bundle over $M_+$, $B^{m_++1}$ is the fiber disc.

Since $S^n_+$ has a metric, we can define a homeomorphism as:
\begin{eqnarray*}
&&B^n_1\sqcup_{id}B^n_2 \longrightarrow S(\nu_+\oplus\textbf{1})\\
&&\qquad \quad e \longmapsto \left\{ \begin{array}{ll}
(e, \sqrt{1-|e|^2}), & \textrm{for $e \in B^n_1$}\\
(e, -\sqrt{1-|e|^2}), & \textrm{for $e \in B^n_2$}
\end{array} \right.
\end{eqnarray*}
where $B^n_1$, $B^n_2$ are two copies of $S^n_+=B(\nu_+)$, $S(\nu_+\oplus \textbf{1})$ is a sphere bundle of the Whitney sum $\nu_+\oplus \textbf{1}$, here $\textbf{1}$ is a trivial line bundle over $M_+$.

As a result, we get a new bundle
\begin{eqnarray*}
&& S^{m_++1} \hookrightarrow D(S^n_+)\\
&&\qquad\qquad\quad \downarrow \rho\\
&&\qquad\quad\quad\quad M_+
\end{eqnarray*}
and a correspondence $D(S^n_+)\cong S(\nu_+\oplus \textbf{1})$. It follows immediately that
\begin{eqnarray*}
T(D(S^n_+))\oplus \textbf{1} &\cong& T(S(\nu_+\oplus \textbf{1}))\oplus \textbf{1} \\
&\cong& \rho^*TM_+\oplus \rho^*(\nu_+\oplus\textbf{1})\\
&\cong& \rho^*{j^*_+}TS^n\oplus\textbf{1}\cong (\textbf{n+1})
\end{eqnarray*}
where $j_+: M_+\rightarrow S^n$ is an inclusion.
In other words, $D(S^n_+)$ is stably parallelizable, \emph{i.e.}, a $\pi$-manifold. This completes the proof.\hfill $\Box$


As indicated in Introduction, we will be mainly concerned with the minimal isoparametric hypersurface $Y$ in the following two cases:
homogeneous hypersurface and FKM-type.
\subsection{Homogeneous hypersurface}

Let $Y$ be a homogeneous hypersurface in $S^n$, as Hsiang and Lawson(\cite{HL}) showed: $Y$
can be characterized as a principal orbit of the isotropy representation of some rank two symmetric space $U/K$.

To begin with, we provide a brief description of the corresponding rank two symmetric spaces. Again, let $g$ be
the number of distinct principal curvatures of the homogeneous hypersurface $Y$. As mentioned before, $g$ can
only be $1$, $2$, $3$, $4$ or $6$.

When $g = 1$, $Y$ is a hypersphere in $S^n$, the corresponding rank two symmetric space is
$$(S^1\times SO(n + 1))/SO(n) = S^1\times S^n.$$

When $g = 2$, $Y$ is a Riemannian product of two spheres $S^k(r)\times S^{n-k-1}(s)$ with $r^2+s^2=1$, $1\leq k\leq n-2$,
the corresponding rank two symmetric space is $$(SO(k + 2)\times SO(n- k + 1))/(SO(k + 1)\times SO(n-k)) = S^{k+1}\times S^{n-k}.$$

When $g = 3$, $Y$ is congruent to a tube of constant radius around the Veronese embedding of real projective
plane $\mathbb{R}P^2$ into $S^4$, or complex projective plane $\mathbb{C}P^2$ into $S^7$, or quaternionic
projective plane $\mathbb{H}P^2$ into $S^{13}$, or Cayley projective plane $\mathbb{O}P^2$ into $S^{25}$.
The corresponding rank two symmetric spaces are
$$SU(3)/SO(3);~~ SU(3)\times SU(3)/SU(3);~~ SU(6)/Sp(3);~~ E_6/F_4.$$

When $g = 4$, $Y$ is a principal orbit of the isotropy representation of
$$SO(5)\times SO(5)/SO(5);~~ SO(10)/U(5);~~ E_6/T\cdot Spin(10);$$ or of two-plane Grassmannians
$$SO(k + 2)/SO(k)\times SO(2)~~(k\geq 3);$$ $$SU(k + 2)/S(U(k)\times U(2))~~(k \geq 3);$$ $$Sp(k + 2)/Sp(k)\times Sp(2)~~(k\geq 2).$$

When $g = 6$, $Y$ is a principal orbit of the isotropy representation of $$G_2/SO(4)~~ or ~~G_2\times G_2/G_2.$$

Now let $G$ be a closed subgroup of the isometry group of $S^n$ acting on $S^n$ with cohomogeneity one. We equip the
orbit space $S^n/G$ with the quotient topology relative to the canonical projection $S^n\rightarrow S^n/G$.
Since $n>1$, $S^n$ is simply connected and compact, for topological reasons $S^n/G$ must be homeomorphic to $[-1, 1]$
and each singular orbit has codimension greater than one.

We denote the singular orbits corresponding to $\pm 1$ by $M_{\pm}$, and their isotropy subgroups by $K_{\pm}$, respectively.
Naturally, there is a diffeomorphism $M_{\pm}\cong G/K_{\pm}$. The other orbits are congruent to each other,
and they are all principal orbits. It makes sense to fix an orbit $Y$ corresponding to a certain value in $(-1,1)$ so that $Y$ is minimal,
and denote its isotropy subgroup by $K_0$. The existence of such $Y$ is clear. Similarly, $Y\cong G/K_0$.

Based on the bundle structure of $S^n_+$ over $M_-$ and the following group action of $K_{\pm}$:
\begin{eqnarray*}
&& K_{\pm}\times(G\times B^{m_++1}_{\pm})\longrightarrow G\times B^{m_++1}_{\pm}\\
&&\qquad\qquad\quad~~(k,g,x)\longmapsto (gk^{-1}, k\bullet x)
\end{eqnarray*}
where $\bullet$ is a slice representation (for details, see for example \cite{Bre}),
we decompose $S^n$ into $$S^n=G\times_{K_+}B^{m_++1}_+\cup_{Y}G\times_{K_-}B^{m_-+1}_-.$$

Next, by gluing two copies of $S^n_+$, we define a new action of the isotropy group $K_+$ on $G\times S^{m_++1}$ :
\begin{eqnarray*}
&&K_+\times(G\times S^{m_++1})\longrightarrow G\times S^{m_++1}\\
&&\qquad(k,g,(x,t))\longmapsto (gk^{-1},k\star(x,t))
\end{eqnarray*}
where $k\star(x,t):=(k\bullet x,t)$, $t=\pm\sqrt{1-|x|^2}$, $(x,t)\in S^{m_++1}$.

Consequently, we have the diffeomorphism: $$D(S^n_+)\cong G\times_{K_+}B^{m_++1}\cup_{Y}G\times_{K_+}B^{m_++1}\cong G\times S^{m_++1}/K_+.$$

In conclusion, $D(S^n_+)$ can be determined by $K_+$ and its action on $G\times S^{m_++1}$. Similarly, we can also express $D(S^n_-)$ in this way.

A series of delicate calculations lead us to a complete list of the
isotropy subgroups $K_0$ and $K_{\pm}$ of homogeneous hypersurfaces
and focal submanifolds as follows. To the best of our knowledge, the
determinations of $K_+$ and $K_-$ have not previously appeared in
the literature. The main difficult occurred in the calculation of
exceptional Lie groups.

\newpage

\begin{center}
Homogeneous (isoparametric) hypersurfaces in the unit sphere.
{\footnotesize

\begin{tabular}{|c|c|l|c|c|c|}
\hline
g & $(m_{+}, m_-)$ & (U, K) & $K_0$ & $K_+$ & $K_-$\\
\hline
\hline
1 & $n-1$ & ($S^1\times SO(n+1)$, $SO(n)$) & $SO(n-1)$ & $SO(n)$ & $SO(n)$\\
  &   & $n\geq 2$    & & &\\
2 & $(p, q)$ & ($SO(p+2)\times SO(q+2)$, & $SO(p)\times SO(q)$ & $SO(p+1)\times SO(q)$  & $SO(p)\times SO(q+1)$\\
  &          & $SO(p+1)\times SO(q+1)$) & & & \\
  &          &  $ p,q\geq 1$ & & &\\
3 & (1, 1)   & $(SU(3), SO(3))$ & $\mathbb{Z}_2+ \mathbb{Z}_2$ & $S(O(2)\times O(1))$  & $S(O(1)\times O(2))$\\
3 & (2, 2)   & $(SU(3)\times SU(3), SU(3))$ & $T^2$ & $S(U(2)\times U(1))$ & $S(U(1)\times U(2))$ \\
3 & (4, 4)   & $(SU(6), Sp(3))$ & $Sp(1)^3$ & $Sp(2)\times Sp(1)$ & $Sp(2)\times Sp(1)$\\
3 & (8, 8)   & $(E_6, F_4)$ & $Spin(8)$ & $Spin(9)$ & $Spin(9)$\\
4 & (2, 2)   & $(SO(5)\times SO(5), SO(5))$ & $T^2$ & $SO(2)\times SO(3)$ & $U(2)$\\
4 & (4, 5)   & $(SO(10), U(5))$ & $SU(2)^2\times U(1)$ & $Sp(2)\times U(1)$ & $SU(2)\times U(3)$\\
4 & (6, 9)   & $(E_6, T\cdot Spin(10))$ & $U(1)\cdot Spin(6)$ & $U(1)\cdot Spin(7) $ & $S^1\cdot SU(5)$\\
4 & (1, m-2) & $(SO(m+2), SO(m)\times SO(2))$ & $SO(m-2)\times \mathbb{Z}_2$ & $SO(m-2)\times SO(2)$ & $O(m-1)$\\
  &          &  $m\geq 3$ & & &\\
4 & (2, 2m-3)& $(SU(m+2), S(U(m)\times U(2)))$ & $S(U(m-2)\times T^2)$ & $S(U(m-2)\times U(2))$ & $S(U(m-1)\times T^2)$\\
  &          &  $m\geq 3$ & & &\\
4 & (4, 4m-5)& $(Sp(m+2), Sp(m)\times Sp(2))$ & $Sp(m-2)\times Sp(1)^2$ & $Sp(m-2)\times Sp(2)$ & $Sp(m-1)\times Sp(1)^2$\\
  &          &  $m\geq 2$ & & &\\
6 & (1, 1)   & $(G_2, SO(4))$ & $\mathbb{Z}_2+ \mathbb{Z}_2$ & $O(2)$ & $O(2)$\\
6 & (2, 2)   & $(G_2\times G_2, G_2)$ & $T^2$ & $U(2)$ & $U(2)$\\
\hline
\end{tabular}}
\end{center}
\vspace{1cm}

In the following, we first illustrate the calculations of $K_0, K_+,
K_-$ in the case of the symmetric pair $(E_6, T\cdot Spin(10))$ with
$(g,m_+,m_-)=(4,6,9)$,  and then give an example of the case
$(SU(3), SO(3))$ with $(g,m_+,m_-)=(3,1,1)$.


\begin{exa}\label{469} \textbf{The calculation of $K_+, K_-$ of the symmetric pair $(E_6, T\cdot Spin(10))$ with $(g,m_+,m_-)=(4,6,9)$.}
\vspace{2mm}

At the beginning, we introduce some notations and operations on the division Cayley algebra $\Ca$,
which is generated by $\{e_0=1, e_1,\cdots,e_7\}$ and satisfies
\begin{enumerate}
  \item For $i>0$, $e_i^2=-1$;
  \item For $i,j>0$, $i\neq j$, $e_ie_j=-e_je_i$;
  \item $e_1e_2=e_4$;
  \item If $e_ie_j=e_k$ for some $i,j,k>0$, then $e_{i+1}e_{j+1}=e_{k+1}$ and $e_{2i}e_{2j}=e_{2k}$ (subscripts $\mod~7$).
\end{enumerate}

Let $M_3(\Ca)$ be the set of $3\times 3$ matrices with entries in $\Ca$,
$\EJ$ the set of Hermitian matrices in $M_3(\Ca)$, namely,  
\[  \EJ=\{ X\in M_3(\Ca) | \kg ^t\overline{X}=X      \}, \]
where the conjugate of any element $x=\sum_{i=0}^7x_ie_i\in \Ca$ is defined by
\begin{equation*}
\overline{x}=x_0e_0-\sum_{i=1}^7x_ie_i.
\end{equation*}

In the following, we will always denote an element $X\in\EJ$ of the form
\[    X=X(\xi,x)=\left(
          \begin{array}{ccc}
            \xi_1 & x_3 & \overline{x}_2 \\
            \overline{x}_3 & \xi_2 & x_1 \\
            x_2 & \overline{x}_1 & \xi_3 \\
          \end{array}
        \right), \kg \mbox{for}~~ \xi_i\in \R,~ x_i\in \Ca
      \]
by
\[   X=\xi_1 E_1+\xi_2 E_2+\xi_3 E_3+F_1(x_1)+F_2(x_2)+F_3(x_3).   \]
The Jordan product, as a basic operation in $\EJ$, is a
multiplication defined by
\[   X\circ Y=\frac{1}{2}(XY+YX), \kg \mathrm{for}~ X,Y\in \EJ.  \]
Usually, $(\EJ, \circ)$ is called the exceptional Jordan algebra. Moreover, the trace $\Tr{X} $, the inner product
$(X,Y)$ and the determinant $\det X$ can be defined respectively by
\begin{eqnarray*}
  &&\Tr{X}=\xi_1+\xi_2+\xi_3, \kg \mathrm{for}~ X=X(\xi,x),\\
  &&(X,Y)=\Tr{X\circ Y},\\
  && \det X=\xi_1\xi_2\xi_3+ \RE(x_1x_2x_3)-\xi_1|x_1|^2-\xi_2|x_2|^2-\xi_3|x_3|^2.
\end{eqnarray*}

Let $\EJ^{\C}=\{X_1+\sqrt{-1}X_2~|~ X_1,X_2\in \EJ\}$ be the complexification of Jordan algebra $\EJ$.
In the same manner, we have the Jordan product, the trace, the $\C$-linear form $(~,~)$ and
the determinant
in $\EJ^{\C}$. A Hermitian inner product $\langle~,~ \rangle$ on $\EJ^{\C}$ is given by
\[   \langle X,Y\rangle=(X,\tau Y),\kg \mbox{for}~X,Y\in \EJ^{\C},  \]
where $\tau$ is the complex conjugate of $\EJ^{\C}$.

With all these notations, an equivalent definition of the group $E_6$ can be given by (\emph{cf}. \cite{Yok})
\[  E_6  = \{ \alpha\in GL(\EJ^{\C},\C)  ~|~ \det (\alpha X)=\det X, \langle\alpha X,\alpha Y\rangle=\langle X,Y\rangle\}.\]
Set
\[  \AH=\{ A\in M_3(\Ca)|\kg ^t\overline{A}=-A,\Tr{A}=0 \}.    \]   
As above, we denote an element $A\in \AH$ of the form
\[ A=\left(
          \begin{array}{ccc}
            a_1 & x_3 & -\overline{x}_2 \\
            -\overline{x}_3 & a_2 & x_1 \\
            x_2 & -\overline{x}_1 & a_3 \\
          \end{array}
     \right), \kg a_i,x_i\in \Ca, \overline{a}_i=-a_i, a_1+a_2+a_3=0.
   \]
by
\[   A=a_1 E_1+a_2 E_2+a_3 E_3+A_1(x_1)+A_2(x_2)+A_3(x_3).   \]

Notice that  $[\AH,\AH]=\AH$, $[\AH,\EJ]=\EJ$. Thus any $A\in \AH$ induces a map $\tilde{A}:\EJ\rightarrow \EJ$ expressed as
\[   \tilde{A}(X)=\frac{1}{2}[A,X].   \]
Let ${\mathfrak{t}}'$ be the subalgebra of $gl(\EJ)$ generated by $\{\tilde{A}|A\in \AH\}$. Then ${\mathfrak{t}}'$
is isomorphic to the (compact) Lie algebra of $F_4$.

Furthermore, observing that any $X\in \EJ$ also induces a map $\tilde{X}:\EJ\rightarrow \EJ$ defined by
\[   \tilde{X}(Y)=X\circ Y,  \kg Y\in \EJ,   \]
we set ${\mathfrak{p}}'=\{\tilde{X}| X\in \EJ, \Tr{X}=0\}$, then the Lie algebra ${\mathfrak{t}}'+\sqrt{-1}{\mathfrak{p}}'$
is just the (compact) Lie algebra of $E_6$, denoted by $\mathfrak{e}_6$. In the following discussions, we will omit the symbol `$\sim$' for simplicity.

Let $\mathfrak{d}_4$ be the subalgebra of  ${\mathfrak{t}}'$ generated by  $\{ \Sigma a_i E_i | a_i\in\Ca, \overline{a}_i=-a_i, \Sigma a_i=0\}$.
Then for any $D \in \mathfrak{d}_4$, $X=X(\xi,x)$,
\[   D \left(
          \begin{array}{ccc}
            \xi_1 & x_3 & \overline{x}_2 \\
            \overline{x}_3 & \xi_2 & x_1 \\
            x_2 & \overline{x}_1 & \xi_3 \\
          \end{array}
        \right)=\left(
                      \begin{array}{ccc}
            0 & D_3x_3 & \overline{D_2x}_2 \\
            \overline{D_3x}_3 & 0 & D_1x_1 \\
            D_2x_2 & \overline{D_1x}_1 & 0 \\
          \end{array}
        \right),
   \]
where $D_1,D_2,D_3$ are elements of Lie algebra $so(8)$ and satisfy the principle of triality:
\[   (D_1 x)y+x(D_2y)=\overline{D_3(\overline{xy})}, \kg x,y\in \Ca,   \]
which implies that $D_2,D_3$ are uniquely determined by $D_1$. Hence the map defined by $D\mapsto D_1$ is an isomorphism from $\mathfrak{d}_4$ to $so(8)$.

Furthermore, setting
\begin{eqnarray*}
&&\mathfrak{D}_i= \{ A_i(x)~|~ x\in \Ca  \}, \kg i=1,2,3,\\
&&\mathfrak{R}_i= \{ F_i(x)~|~ x\in \Ca  \}, \kg i=1,2,3,\\
&&\mathfrak{R}_0= \{ \Sigma \xi_iE_i ~|~ \xi_i\in \R, \Sigma\xi_i=0  \},
\end{eqnarray*}
we can decompose the Lie algebra $\mathfrak{e}_6$ as
\begin{equation*}
\mathfrak{e}_6=\mathfrak{d}_4+\mathfrak{D}_1+\mathfrak{D}_2+\mathfrak{D}_3
                     +\sqrt{-1}\mathfrak{R}_0+\sqrt{-1}\mathfrak{R}_1+\sqrt{-1}\mathfrak{R}_2+\sqrt{-1}\mathfrak{R}_3.
\end{equation*}
Since there is a transformation $\sigma$ of $\EJ^{\C}$ expressed as
\[   \sigma \left(
          \begin{array}{ccc}
            \xi_1 & x_3 & \overline{x}_2 \\
            \overline{x}_3 & \xi_2 & x_1 \\
            x_2 & \overline{x}_1 & \xi_3 \\
          \end{array}
        \right)=\left(
          \begin{array}{ccc}
            \xi_1 & -x_3 & -\overline{x}_2 \\
            -\overline{x}_3 & \xi_2 & x_1 \\
            -x_2 & \overline{x}_1 & \xi_3 \\
          \end{array}
        \right),   \]
( obviously, $\sigma^2=\emph{id}$ ), an involution $\gamma$ of
$E_6$ can be naturally defined by
$\gamma(\alpha)=\sigma \alpha \sigma,  $ for $\alpha\in E_6$.
Thus the decomposition of $\mathfrak{e}_6$ corresponding to $\gamma$ can be written as $\mathfrak{e}_6=\mathfrak{t}+\mathfrak{p}$, where
\begin{eqnarray*}
  \mathfrak{t}&=&\{\delta\in e_6~~| ~~\sigma \delta =~~\delta \sigma\}=\mathfrak{d}_4+\mathfrak{D}_1+\sqrt{-1}\mathfrak{R}_0+\sqrt{-1}\mathfrak{R}_1,\\
  \mathfrak{p}&=&\{\delta\in e_6~~| ~~ \sigma \delta =-\delta \sigma\}=\mathfrak{D}_2+\mathfrak{D}_3+\sqrt{-1}\mathfrak{R}_2+\sqrt{-1}\mathfrak{R}_3.
\end{eqnarray*}
Choosing a maximal Abelian subspace of $\mathfrak{p}$ as $\mathfrak{h}=\{A_2(\lambda_1e_0)+\sqrt{-1}F_2(\lambda_2e_1)~|~\lambda_1,\lambda_2\in \R\}$, and denoting by $\Delta$ the set of restricted positive roots with respect to $\mathfrak{h}$, we have
\begin{equation}\label{e6}
\mathfrak{e}_6=\mathfrak{m}+\mathfrak{h}+\sum_{\lambda\in \Delta}\{\mathfrak{t}_{\lambda}+\mathfrak{p}_{\lambda}\},
\end{equation}
where
\begin{eqnarray*}
&&\mathfrak{m}=\{A\in \mathfrak{t}~| \kg [A,H]=0,  \mbox{for~}  H\in \mathfrak{h}\},\\
&&\mathfrak{t}_{\lambda}= \{  A\in \mathfrak{t}~| \kg ad(H)^2A=-\lambda(H)^2A, \mbox{for~} H\in \mathfrak{h} \},\\
&&\mathfrak{p}_{\lambda}= \{  A\in \mathfrak{p}~| \kg ad(H)^2A=-\lambda(H)^2A, \mbox{for~} H\in \mathfrak{h} \}.
\end{eqnarray*}
Set $\widetilde{e}_i=e_1e_i$, for $i>1$ and $G_{ij}=E_{ij}-E_{ji}$, for $i,j=0,1,\cdots,7$, where $E_{ij}$ is the matrix with $(i,j)$ entry $1$ and all others $0$. By a direct computation,
we can express $\mathfrak{m}$ explicitly as
\begin{equation*}
\mathfrak{m}=\mbox{span}\{\sqrt{-1}(E_1-2E_2+E_3), D~|~ D_2=G_{ij},  i,j>1\}\cong so(6)\oplus \R.
\end{equation*}
Moreover, we calculate $\mathfrak{t}_{\lambda}$ and $\mathfrak{p}_{\lambda}$ in (\ref{e6}) with respect to the root system $\Delta$, and list them in the following table.

\begin{center}
\begin{tabular}{|c|c|c|c|}
\hline
$\mathrm{dimension}$ & $\lambda\in \Delta$ & basis of $\mathfrak{t}_{\lambda}$ & basis of $\mathfrak{p}_{\lambda}$\\
\hline
\hline
$6$  &  $\lambda_1$ & $D\in \mathfrak{d}_4 : D_2=G_{0i}, i>1 $ & $ A_2(e_i) :i>1$\\
\hline
$6$  &  $\lambda_2$ & $D\in \mathfrak{d}_4 : D_2=G_{1i}, i>1$ & $ \sqrt{-1}F_2(e_i): i>1$\\
\hline
$1$  &  $\lambda_1-\lambda_2$ & $D+\sqrt{-1}(E_1-E_3) : D_2=G_{01}$ & $ A_2(e_1)+\sqrt{-1}F_2(e_0)$\\
\hline
$1$  &  $\lambda_1+\lambda_2$ & $D-\sqrt{-1}(E_1-E_3) : D_2=G_{01}$ & $ A_2(e_1)-\sqrt{-1}F_2(e_0)$\\
\hline
 &   & $A_1(e_0)-\sqrt{-1}F_1(e_1)$ & $ A_3(e_0)-\sqrt{-1}F_3(e_1)$\\
$8$ & $\frac{1}{2}(\lambda_1-\lambda_2)$  & $A_1(e_1)+\sqrt{-1}F_1(e_0)$ & $ A_3(e_1)+\sqrt{-1}F_3(e_0)$\\
 &   & $A_1(e_i)+\sqrt{-1}F_1(\widetilde{e}_i): i>1$ & $ A_3(e_i)-\sqrt{-1}F_3(\widetilde{e}_i):i>1$\\
\hline
 &   & $A_1(e_0)+\sqrt{-1}F_1(e_1)$ & $ A_3(e_0)+\sqrt{-1}F_3(e_1)$\\
$8$ & $\frac{1}{2}(\lambda_1+\lambda_2)$  & $A_1(e_1)-\sqrt{-1}F_1(e_0)$ & $ A_3(e_1)-\sqrt{-1}F_3(e_0)$\\
 &   & $A_1(e_i)-\sqrt{-1}F_1(\widetilde{e}_i):i>1$ & $ A_3(e_i)+\sqrt{-1}F_3(\widetilde{e}_i):i>1$\\

\hline
\end{tabular}
\end{center}
\vspace{3mm}

Let $K=\{\alpha\in E_6 ~|~ \alpha\sigma=\sigma\alpha\}$, which acts on $\mathfrak{p}$ by the adjoint representation. The
orbits can only be of the following three types:

$1^{\circ}$ If $H_0\in \mathfrak{h}$ with $\lambda_1(H_0)\cdot\lambda_2(H_0)\neq 0$ and $\lambda_1(H_0)\neq \pm \lambda_2(H_0)$,
the Lie algebra of the isotropy subgroup $~K_0$ at $H_0$ is $\mathfrak{m} $ and
$  K_0\cong U(1)\cdot Spin(6).   $

$2^{\circ}$ If $H_+\in \mathfrak{h}$ with either $\lambda_1(H_+)= 0$ or $\lambda_2(H_+)= 0$. Without loss of generality,
assume $\lambda_2(H_+)= 0$.
According to the previous table,
the Lie algebra of the isotropy subgroup $K_+$ is $\mathfrak{m}\oplus \mathfrak{k}_{\lambda_2}\cong so(7)\oplus \R. $
Then it is not difficult to see that
$$    K_+\cong U(1)\cdot Spin(7).         $$

$3^{\circ}$ If $H_-\in \mathfrak{h}$ with either $\lambda_1(H_-)=\lambda_2(H_-)$ or $\lambda_1(H_-)=-\lambda_2(H_-)$.
Without loss of generality,
assume $\lambda_1(H_-)=\lambda_2(H_-)$. According to the previous table, the Lie algebra $\mathfrak{t}_-$ of the
isotropy subgroup $K_-$ is given by
\[   \mathfrak{t}_-=\mathfrak{m}+\mathfrak{t}_{\mu}+\mathfrak{t}_{\mu/2}, \kg \mbox{for~} \mu=\lambda_1-\lambda_2.  \]

Let $V^{10}$ be a $10$-dimensional vector space defined by
\[   V^{10}=\bigg\{\left(
          \begin{array}{ccc}
            0 & 0 & 0 \\
            0 & \xi & x \\
            0 & \overline{x} & -\tau(\xi) \\
          \end{array}
        \right) \kg |\kg \xi\in \C, x\in \Ca  \bigg\}\subset\EJ^{\C}.   \]
It is well known that $K\cong T^1\cdot Spin(10)$, and the representation $\phi:Spin(10)\rightarrow SO(V^{10})$ is just the vector representation.

We finally introduce a complex structure $J$ on $V^{10}$ as follows,
\[    J\left(
          \begin{array}{ccc}
            0 & 0 & 0 \\
            0 & \xi & x \\
            0 & \overline{x} & -\tau(\xi) \\
          \end{array}
        \right)= \left(
          \begin{array}{ccc}
            0 & 0 & 0 \\
            0 & -\sqrt{-1}\xi & x \cdot e_1 \\
            0 & \overline{x\cdot e_1} & -\sqrt{-1}\tau(\xi) \\
          \end{array}
        \right), \kg \mbox{for}~~\xi\in \C, x\in \Ca.      \]

By a direct computation, we find that elements of the subalgebra $[\mathfrak{t}_-,\mathfrak{t}_-]\subset so(10)$ commute with the complex structure $J$, specifically,
$[\mathfrak{t}_-,\mathfrak{t}_-]\cong su(5)$. Moreover, the center $\mathfrak{c(t_-)}\cong \R$ is not contained in $so(10)$. Therefore,
we can conclude
\[  \mathfrak{t}_-=\mathfrak{c(t_-)}\oplus [\mathfrak{t}_-,\mathfrak{t}_-]\cong \R\oplus su(5),   \]
and via the representation $\phi$ the corresponding isotropy subgroup is
\[  K_-\cong S^1\cdot SU(5),    \]
where $S^1$ is a group generated by the center $\mathfrak{c(t_-)}$.
\end{exa}
\vspace{5mm}


\begin{exa}\label{311}
\textbf{An explicit description of $D(S^4_+)$ with $(g,m_+, m_-)=(3,1,1)$.}
\vspace{2mm}

Firstly, recall a result of E. Cartan that the isoparametric hypersurface in this case must be a tube of constant radius over a standard
Veronese embedding of $\mathbb{R}P^2$ into $S^4$.

Let $\nu$ be the normal bundle of $\mathbb{R}P^2\hookrightarrow
S^4$, which is non-orientable since $T\mathbb{RP}^2 \oplus \nu =
\textbf{4}$, a $4$-dimensional trivial bundle. Let $\eta$ be the
Hopf line bundle over $\mathbb{R}P^2$. It is well known that
$T\mathbb{R}P^2\oplus\textbf{1}=3\eta$. Thus
$3\eta\oplus\nu=T\mathbb{R}P^2\oplus\textbf{1}\oplus
\nu=\textbf{5}$. Hence $4\eta\oplus\nu$=$\textbf{5}\oplus\eta$.
\vspace{2mm}

\noindent
\textbf{Assertion 1:} $4\eta\cong \textbf{4}$.
\vspace{2mm}

It follows at once that $\nu\oplus \textbf{4}=\eta\oplus\textbf{5}$. Then we deduce by obstruction theory that
$\nu\oplus\textbf{1}=\eta\oplus\textbf{2}$, and thus $D(S^4_+)\cong S(\nu\oplus \textbf{1})\cong S(\eta\oplus \textbf{2})$.
Furthermore, we show
\vspace{2mm}

\noindent
\textbf{Assertion 2:} $D(S^4_+)\cong S^2\times S^2/\sigma$, where $\sigma$ is an involution.
\vspace{2mm}

\noindent
\emph{Proof of Assertion 2:}\label{assertion 2}
Again, let $\eta$ be a Hopf line bundle over $\mathbb{R}P^n$, $E(\eta)$ be the total space of $\eta$, then
\begin{eqnarray}\label{eta}
&&E(\eta)\cong S^n\times \mathbb{R}\Big/(x,t)\sim(-x,-t)\\
&&\qquad\quad \downarrow\nonumber\\
&&\qquad S^n/x\sim-x=\mathbb{R}P^n\nonumber
\end{eqnarray}
where $x\in S^n$, $t\in \mathbb{R}$.
This interpretation deduces that for $x\in S^n$, $(t_1,...,t_p)\in \mathbb{R}^p$, and $(s_1,...,s_q)\in \mathbb{R}^q$,
\begin{equation}\label{p+qeta}
S^n\times \mathbb{R}^{p+q}\Big/(x,t_1,...,t_p,s_1,...,s_q)\sim(-x,t_1,...,t_p,-s_1,...,-s_q)\cong E(\textbf{p}\oplus q\eta)
\end{equation}
In particular, $$D(S^4_+)\cong S(\eta\oplus\textbf{2})\cong S^2\times S^2\Big/(x,y_1,y_2,y_3)\sim(-x,-y_1,y_2,y_3).$$
where $x\in S^2$, $(y_1,y_2,y_3)\in S^2$.\hfill $\Box$
\vspace{6mm}

Assertion $1$ should be well known. However, we would like to give an interesting and simple proof.

\noindent
\emph{Proof of Assertion 1:}
By (\ref{p+qeta}), it suffices to define a pointwise isomorphism $\Phi$
\begin{eqnarray}\label{assertion 1}
&& S^2\times \mathbb{R}^4  \Big/ (x, t)\sim (-x, -t) ~~\xrightarrow{\Phi}~~ S^2/\mathbb{Z}_2\times \mathbb{R}^4\\
&&\qquad\qquad\qquad\qquad \downarrow \qquad\qquad\qquad\qquad \downarrow\nonumber\\
&&\qquad\qquad\qquad\quad \mathbb{R}P^2\qquad\quad \xrightarrow{\emph{id}}\qquad \mathbb{R}P^2\nonumber
\end{eqnarray}
where $\mathbb{R}^4$ is identified with the quaternions $\mathbb{H}$, and $x\in S^2=\{x\in \mathbb{H}~|~|x|=1, Re~x=0\}$, $t\in\mathbb{H}$.

Define $\Phi (x,t):= (x, xt)$. Obviously, $\Phi$ is well-defined, and for a fixed $x$, it is a linear isomorphism.\hfill $\Box$
\vspace{6mm}

It is worth remarking that $D(S^4_+)$ is not diffeomorphic to the oriented Grassmannian
$$G_2(\mathbb{R}^4)\cong S^2\times S^2/(x,y)\sim(-x,-y),$$
where $x, y\in S^2$. To show this remark, firstly, recall that (\emph{cf.} \cite{Tan})
$$H^*(G_2(\mathbb{R}^4);\mathbb{Z}_2)\cong \mathbb{Z}_2[a_1,a_2]\Big/a^3_1=0, a^2_1a_2+a^2_2=0,$$
where $a_1\in H^1(G_2(\mathbb{R}^4);\mathbb{Z}_2)$, $a_2\in H^2(G_2(\mathbb{R}^4);\mathbb{Z}_2)$.
By this cohomology ring structure, it is not difficult to conclude the total Stiefel-Whitney class $$W(G_2(\mathbb{R}^4))=1+w^2_1.$$
This means that the first Stiefel-Whitney
class vanishes, and the second Stiefel-Whitney class $w_2(G_2(\mathbb{R}^4))\neq 0$. In other words, $G_2(\mathbb{R}^4)$
is an orientable manifold, but not spin,
while $D(S^4_+)$ is spin as mentioned in Proposition \ref{pi-manifold}.
\vspace{6mm}

\end{exa}


\subsection{FKM-type}

In this subsection, we investigate the FKM-type isoparametric hypersurfaces in spheres with four distinct
principal curvatures (\emph{cf.} \cite{FKM} and \cite{CCJ}).

According to \cite{FKM}, for a symmetric Clifford system $\{P_0,...,P_m\}$
on $\mathbb{R}^{2l}$, \emph{i.e.}, $P_i$'s are symmetric matrices
satisfying $P_iP_j+P_jP_i=2\delta_{ij}I_{2l}$, there is a homogeneous polynomial on $\mathbb{R}^{2l}$ defined by
\begin{eqnarray}\label{FKM isop. poly.}
&& \quad F: \mathbb{R}^{2l} \rightarrow \mathbb{R}\\
&& F(z) = |z|^4 - 2\displaystyle\sum_{i = 0}^{m}{\langle
P_iz,z\rangle^2}.\nonumber
\end{eqnarray}
It can be shown that if $l-m-1>0$, then the level sets of the restriction $f=F|_{S^{2l-1}}$ constitute a family
of isoparametric hypersurfaces with
$g=4$ distinct principal curvatures with multiplicities $m_+=m$, $m_-=l-m-1$. The
focal submanifolds are $M_+=f^{-1}(1)$, $M_-=f^{-1}(-1)$, with codimensions $m+1$ and $l-m$ in $S^{2l-1}$, respectively.

Clearly, the $+1$ eigenspace of $P_0$, say $E_+(P_0)$, is invariant
under the transformations $E_1=P_1P_2$,..., $E_{m-1}=P_1P_m$. As
usual, let $\delta(m)$ be the dimension of the irreducible Clifford
algebra $\mathcal{C}_{m-1}$-modules (\emph{e.g.}, $\delta(4)=4$,
$\delta(8)=8$, $\delta(m+8)=16\delta(m)$). Then $l=k\delta(m)$, for
some positive integer $k$. As is known in representation theory,
when $m\equiv 0$ $(\mod~4)$, there exist exactly two irreducible
$\mathcal{C}_{m-1}$-modules $\Delta^+_m$ and $\Delta^-_m$
distinguished by $E_1E_2\cdots E_{m-1}=Id~~or-Id.$ If we write
$E_+(P_0)=a\Delta^+_m\oplus b\Delta^-_m$ as
$\mathcal{C}_{m-1}$-modules, then $$tr(P_0P_1\cdots P_m) =
2q\delta(m),$$ where $q=a-b.$ On the other hand, noticing $k=a+b$,
we see
\begin{equation}
q\equiv k~(\mod~2).
\end{equation}

By \cite{Wan}, two symmetric Clifford systems with the same index
$q$ give rise to equivalent isoparametric functions on $S^{2l-1}$.
Therefore, when $m\equiv 0$ $(\mod~4)$, $f$ is determined by $m$,
$l=k\delta(m)$, as well as $q$ up to a rigid motion of $S^{2l-1}$.
However, when $m\not \equiv 0 $ $(\mod~4)$, $q$ is always zero, so
$f$ is determined only by $m$ and $l$ up to a rigid motion of
$S^{2l-1}$.

In the rest of this subsection, we focus on the case when $m\equiv
0$ $(\mod~4)$.

First, we denote the corresponding isoparametric hypersurface by $M(m,l,q)$, and focal submanifolds by $M_{\pm}(m,l,q)$.
Next, recall the following conclusions shown in (\cite{FKM}):

$(a)$~The normal bundle $\nu_+$ is trivial, in particular, $M(m,l,q)$ is diffeomorphic to the product $M_+(m,l,q)\times S^m$;

$(b)$~$M_-(m,l,q)$ is diffeomorphic to an $S^{l-1}$ bundle over $S^m$.


\vspace{3mm}

Therefore, the property $(a)$, together with the proof of Proposition \ref{pi-manifold}, implies that $D(S^{2l-1}_+)\cong D(B(\nu_+))\cong S(\nu_+\oplus \textbf{1}) \cong M_+\times S^{m+1}$. In the
remaining part, we will be concerned with the topology of $D(S^{2l-1}_-)$. On this account, we prove Theorem \ref{FKM} in two steps.
\vspace{3mm}

\noindent
\textbf{Step 1:}
As stated before, $D(S^{2l-1}_-)\cong D(B(\nu_-))\cong S(\nu_-\oplus \textbf{1})$ is the total space of a sphere bundle.
Set $\zeta:=\nu_-\oplus\textbf{1}$, then $D(S^{2l-1}_-)=S(\zeta)$, where $\zeta$ is a vector bundle over $M_-$ with disc bundle $B(\zeta)$
and sphere bundle $S(\zeta)$, respectively.
\begin{eqnarray}\label{zeta}
&&\zeta: \mathbb{R}^{l-m+1} \hookrightarrow E(\zeta)\supset B(\zeta)\supset S(\zeta) = D(B(\nu_-))\\
&&\qquad\qquad\qquad\quad \downarrow \pi\nonumber\\
&&\qquad\qquad\quad\quad M^{m+l-1}_-\nonumber
\end{eqnarray}

As mentioned in $(b)$ above, $M_-(m,l,q)$ is diffeomorphic to an $S^{l-1}$ bundle over $S^m$, that is:
$M_-(m,l,q)\cong S(\xi)$, where $\xi$ is a vector bundle over $S^m$ so that
\begin{eqnarray}\label{M-}
&& S^{l-1} \hookrightarrow M_-=S(\xi)\\
&&\qquad\qquad \downarrow \pi_1\nonumber\\
&&\qquad\qquad S^m\nonumber
\end{eqnarray}

\begin{lem}\label{3.1}
The Pontrjagin class $$p_{\frac{m}{4}}(\zeta)=-\pi^*_1p_{\frac{m}{4}}(\xi)=-q\cdot\beta(m)\cdot (\frac{m}{2}-1)!\cdot \pi_1^*\gamma,$$ where
$\beta(m)=\left\{\begin{array}{c c}
1, & m\equiv 0~ (\mod~8) \\
2, & m\equiv 4~ (\mod~8)
\end{array}\right.$, $\gamma \in H^m(S^m; \mathbb{Z})$ is a suitable generator.
\end{lem}

\begin{proof}

Denoting the total Pontrjagin class of an arbitrary vector bundle $\eta$ by $P(\eta)=1+p_1(\eta)+p_2(\eta)+...$, from \cite{MS}, we
know that $P(\eta\oplus\textbf{1})=P(\eta)$.

Firstly, since $\nu_-\oplus TM_-\cong TS^{2l-1}|_{M_-}$ is stably
parallelizable (in fact it is trivial for the dimension reason), we
have $P(\nu_-\oplus TM_-)=1$. On the other hand, since $m\geq 4$,
$l=k\delta(m)\equiv 0~(\mod~4)$, thus $l-1$ can not be divided by
$4$. By reason of rank of the sphere bundle $S(\xi)$, we can deduce
that $P(M_-):=P(TM_-)=1+p_{\frac{m}{4}}(M_-)$.


Consequently, since in this case the cohomology of $M_-$ with
coefficients $\mathbb{Z}$ has no torsion ( Proposition
\ref{homology} ), it follows from $P(\zeta)=P(\nu_-\oplus
\textbf{1})=P(\nu_-)$ that
\begin{equation}\label{p4m}
p_{\frac{m}{4}}(\zeta)=-p_{\frac{m}{4}}(M_-).
\end{equation}

Next, for the tangent bundle of $M_-$, we have $TM_-\oplus\textbf{1}\cong \pi_1^*TS^m\oplus\pi_1^*\xi$. As a direct result,
$TM_-\oplus\textbf{2}=\textbf{(m+1)}\oplus\pi_1^*\xi$, which implies that
$p_{\frac{m}{4}}(M_-)=p_{\frac{m}{4}}(\pi^*_1\xi)$.

Since $l-1>m$, there exists a section of the sphere bundle $S(\xi)$,
thus its Euler class $e$ vanishes. Hence, from the Gysin cohomology
sequence with coefficients $\mathbb{Z}$ associated with $S(\xi)$:
\begin{equation}\label{Gysin}
\rightarrow H^i(S^m)\xrightarrow{e}H^{i+l}(S^m)\xrightarrow{\pi_1^*}H^{i+l}(M_-)\rightarrow H^{i+1}(S^m)\rightarrow\cdots,
\end{equation}
we deduce that
\begin{equation}\label{pi1 isomorphism}
\pi^*_1:~~H^m(S^m; \mathbb{Z})\longrightarrow H^m(M_-; \mathbb{Z})~\qquad~~~~is~ an ~isomorphism.
\end{equation}
Thus by (\ref{p4m})
\begin{equation}\label{p4m xi}
p_{\frac{m}{4}}(\zeta)=-\pi^*_1p_{\frac{m}{4}}(\xi)~\quad under~ the~ isomorphism ~\pi^*_1.
\end{equation}


At the mean time, $\xi-rank~\xi \in \widetilde{KO}(S^m)$, which will be abbreviated as $\xi$. Let us consider the complexification homomorphism
\begin{eqnarray}\label{c}
&& \mathbb{C}:~\widetilde{KO}(S^m)~\rightarrow~\widetilde{K}(S^m)\\
&&\qquad\quad \quad\xi ~~\mapsto~~~\xi\otimes \mathbb{C}\nonumber\\
&&\qquad\quad \quad 1 ~~\mapsto~~~\left\{\begin{array}{c c}
1, & m\equiv 0~ (\mod~8) \\
2, & m\equiv 4~ (\mod~8)
\end{array}\right.\nonumber
\end{eqnarray}


Recall a well known result that the Chern character $Ch:~\widetilde{K}(S^m)\longrightarrow H^m(S^m; \mathbb{Z})$ is an isomorphism for even $m$, namely, the top
Chern class of the generator of $\widetilde{K}(S^m)$ is equal to a generator of $H^m(S^m; \mathbb{Z})$ multiplied by $(\frac{m}{2}-1)!$.
By the isomorphisms $\xi\cong a\Delta^++b\Delta ^-\cong (a-b)\Delta^++b(\Delta^++\Delta ^-)\cong q\Delta^++\textbf{b}$,
(since $\Delta^++\Delta ^-$ is trivial) where $\Delta^+\in \widetilde{KO}(S^m)$ is a generator, we finally arrive at
\begin{equation}\label{delta}
p_{\frac{m}{4}}(\xi)= q\cdot\beta(m)\cdot (\frac{m}{2}-1)!\cdot \gamma \in H^m(S^m; \mathbb{Z})\cong \mathbb{Z},
\end{equation}
where $\gamma\in H^m(S^m; \mathbb{Z})$ is a suitable generator, and
$\beta(m)=\left\{\begin{array}{c c}
1, & m\equiv 0~ (\mod~8) \\
2, & m\equiv 4~ (\mod~8)
\end{array}\right.$


\vspace{4mm}

In summary, $p_{\frac{m}{4}}(\zeta)=-\pi^*_1p_{\frac{m}{4}}(\xi)=-q\cdot\beta(m)\cdot (\frac{m}{2}-1)!\cdot \pi_1^*\gamma$.
\end{proof}
\vspace{5mm}

\noindent
\textbf{Step 2:} We mainly make use of Wu Square $\mathcal{P}^1$ (\emph{cf}. \cite{MS} ).

Let $p:=2r+1\geq 3$ be an odd prime, and $m=2(p-1)=4r\equiv
0~(\mod~4)$. We will need the following fundamental criterion from
number theory.

\noindent \textbf{Wilson's Theorem:}\,\, {\itshape $p$ is a
prime\quad if and only if\quad $(p-1)!\equiv-1~(\mod~p)$. }

We recall the Wu Squares with coefficients $\mathbb{Z}_p$, which are
generalized from Steenord Squares with coefficients $\mathbb{Z}_2$
by Wu Wen-Ts\"{u}n:
$$\mathcal{P}^i:~~H^j(X; \mathbb{Z}_p)\rightarrow H^{j+4ri}(X; \mathbb{Z}_p).$$
For $(B(\zeta), S(\zeta))\supset B(\zeta) \supset S(\zeta)$, one has
the following commutative diagram of the cohomology sequences with
coefficients $\mathbb{Z}_p$:
\begin{eqnarray}\label{delta Pi}
&& \quad H^j(B(\zeta), S(\zeta))~~\rightarrow ~~ H^j(B(\zeta))~~ \rightarrow~~ H^j(S(\zeta))~~\xrightarrow{\delta}~~ H^{j+1}(B(\zeta), S(\zeta))\rightarrow \cdots\\
&&\qquad\quad\downarrow \mathcal{P}^i\qquad\qquad\qquad \downarrow \mathcal{P}^i\qquad\qquad\quad \downarrow \mathcal{P}^i\qquad\qquad\quad \downarrow \mathcal{P}^i\nonumber\\
&&H^{j+4ri}(B(\zeta), S(\zeta))\rightarrow H^{j+4ri}(B(\zeta))\rightarrow H^{j+4ri}(S(\zeta))\xrightarrow{\delta} H^{j+1+4ri}(B(\zeta), S(\zeta))\rightarrow\cdots\nonumber
\end{eqnarray}
in particular, $\mathcal{P}^i$ satisfies $\delta\mathcal{P}^i=\mathcal{P}^i\delta$.

Define $q_i(\zeta):=\Phi^{-1}\cdot \mathcal{P}^i\cdot \Phi(1)$, where $1\in H^0(M_-; \mathbb{Z}_p)$, and $\Phi$ is the Thom isomorphism
$$\Phi: H^i(M_-; \mathbb{Z}) \longrightarrow H^{i+l-m+1}(B(\zeta), S(\zeta); \mathbb{Z}).$$ By Wu Theorem (\emph{cf.} \cite{MS}), $q_i(\zeta)$ can be expressed in form of a combination of Pontrjagin classes $p_1(\zeta)$, $p_2(\zeta)$,..., $p_{ri}(\zeta)$. Observing $p_0(\zeta)=1$, $p_1(\zeta)=0$,...,$p_{r-1}(\zeta)=0$, $p_r(\zeta)\neq 0$, we want to represent $q_1(\zeta)$ by $p_r(\zeta)$.

By Newton's identities and Lemma \ref{3.1},
\begin{eqnarray*}
&&q_1(\zeta)\\
&\equiv& (-1)^{r+1}\cdot r\cdot p_r(\zeta)~(\mod~p)\\
&\equiv& (-1)^{\frac{m}{4}+1}\cdot r\cdot (-1)\cdot q \cdot \beta(m)\cdot (\frac{m}{2}-1)!\pi_1^*(\gamma)~ (\mod~p)\\
&\equiv& (-1)^{\frac{m}{4}}\cdot q\cdot
\beta(m)\cdot\frac{p-1}{2}\cdot(p-2)!\pi_1^*(\gamma)~ (\mod~p)
\end{eqnarray*}

When $m\equiv 4 ~(\mod~8)$, a direct application of Wilson Theorem
gives $q_1(\zeta)\equiv (-1)^{\frac{m}{4}+1}\cdot q\cdot
\pi_1^*(\gamma)~ (\mod ~p)$; When $m\equiv 0 ~(\mod~8)$, by Wilson
Theorem, it is not difficult to show $\frac{1}{2}(p-1)!\equiv
\frac{1}{2}(p-1)$ $(\mod~p)$,
which yields that $q_1(\zeta)\equiv (-1)^{\frac{m}{4}}\cdot q\cdot\frac{p-1}{2}\cdot \pi_1^*(\gamma)~ (\mod ~p)$. 

In summary,
\begin{equation}\label{p1 phi}
q_1(\zeta)\equiv\left\{
\begin{array}{c c}
(-1)^{\frac{m}{4}}\cdot q\cdot\frac{p-1}{2}\pi_1^*(\gamma)~ (\mod ~p), & m\equiv 0 ~(\mod~8)\\
(-1)^{\frac{m}{4}+1}\cdot q\cdot \pi_1^*(\gamma)~ (\mod ~p),&
m\equiv 4 ~(\mod~8)
\end{array}\right.
\end{equation}

Fix $j=l-m$, $i=1$ in (\ref{delta Pi}). Since $m<l-m<l-1$, $H^{l-m}(B)=0$, $H^l(B)=0$, so both $\delta$'s are injective:
\begin{eqnarray}
&&H^{l-m}(S)  \autorightarrow{$\delta$}{inj.} H^{l-m+1}(B,S) \autoleftarrow {$\Phi$}{$\cong$} H^0(M_-)\cong \mathbb{Z}_p\\
&&\quad\downarrow\mathcal{P}^1\qquad\qquad\quad~~ \downarrow\mathcal{P}^1\nonumber\\
&&~~~H^{l}(S)  ~~\autorightarrow{$\delta$}{inj.}~~ H^{l+1}(B,S) \autoleftarrow {$\Phi$}{$\cong$} H^m(M_-)\cong \mathbb{Z}_p\nonumber
\end{eqnarray}

On the other hand, by Gysin cohomology sequence of $\zeta$ with coefficient $\mathbb{Z}_p$:
\begin{equation}\label{Gysin zeta}
\rightarrow H^k(M_-)\xrightarrow{\cup e}H^{k+l-m+1}(M_-)\xrightarrow{\pi^*}H^{k+l-m+1}(S(\zeta))\rightarrow H^{k+1}(M_-)\rightarrow\cdots
\end{equation}
we get $H^{l-m}(S(\zeta))\cong \mathbb{Z}_p$, $H^l(S(\zeta))\cong \mathbb{Z}_p$, thus $\delta$ are isomorphisms since they are injective as we stated above.
At last, by the definition $q_1(\zeta):=\Phi^{-1}\cdot \mathcal{P}^1\cdot \Phi(1)$, different values of $q$ give rise to different Wu squares $\mathcal{P}^1$, as we desired.

The proof of Theorem \ref{FKM} is now complete!

\begin{ack}
The authors would like to thank Professor Hui Ma for her valuable
conversations concerning the calculations on exceptional Lie groups.
They also want to express their gratitude to Professor Kefeng Liu
and the referee for their interest and helpful comments.
\end{ack}

\end{document}